\documentclass[12pt]{amsart}
\usepackage{amsfonts}
\usepackage{tikz}
\usepackage{ifthen}
\usepackage{amsmath}
\usepackage{amsthm}
\usepackage{amssymb}
\usepackage{tikz-cd} 
\usepackage{bbm}
\theoremstyle{plain}
\newtheorem{theorem}{\bf{Theorem}}[section]
\newtheorem{lemma}[theorem]{Lemma}
\newtheorem{proposition}[theorem]{Proposition}
\newtheorem{corollary}[theorem]{Corollary}
\newtheorem{remark}[theorem]{Remark}

\newtheorem{definition}[theorem]{Definition}

\usepackage[top=.5 in,bottom=.5 in, left=.5 in, right= .5 in]{geometry}

\newcommand{\bX}{\mathbb{X}}
\newcommand{\bY}{\mathbb{Y}}
\newcommand{\bZ}{\mathbb{Z}}
\newcommand{\bV}{\mathbb{V}}
\newcommand{\bU}{\mathbb{U}}
\newcommand{\bI}{\mathbb{I}}
\newcommand{\bT}{\mathbb{T}}
\newcommand{\bR}{\mathbb{R}}
\newcommand{\bRx}{\left(z\bI-\bX\right)^{-1}}
\newcommand{\gxy}{G_{\bX+\bY}(z)}

\newcommand{\gx}{G_{\bX}(z)}

\newcommand{\oj}{\omega_1(z)}
\newcommand{\od}{\omega_2(z)}
\newcommand{\cgw}{\textrm{C}^\ast}
\newcommand{\calA}{\mathcal{A}}

\numberwithin{equation}{section}

\begin{document}
	\title[The Matsumoto-Yor property in free probability]{The Matsumoto-Yor property in free probability via subordination and Boolean cumulants}

	\author[M.\'{S}wieca]{Marcin \'{S}wieca}
	\address{Wydzia{\l} Matematyki i Nauk Informacyjnych\\
		Politechnika Warszawska\\
		ul. Ko\-szy\-ko\-wa 75\\
		00-662 Warsaw, Poland} \email{M.Swieca@mini.pw.edu.pl}
			
	\keywords{}
	
	\date{\today}

	\begin{abstract}
We study the Matsumoto-Yor property in free probability. We prove three characterizations of free-GIG and free Poisson distributions by freeness properties together with some assumptions about conditional moments. Our main tools are subordination and Boolean cumulants. In particular, we establish a new connection between additive subordination function and Boolean cumulants.
	\end{abstract}
	
	\maketitle
\section{Introduction}

 In \cite{matsumoto2001analogue} authors observed an interesting property of Gamma and Generalized Inverse Gaussian (GIG) laws that is now known in literature as the Matsumoto-Yor property: If $X$ has the Generalized Inverse Gaussian law  $GIG(-p,a,b)$, $Y$ has the  Gamma law  $G(p,a)$ and  $X$ and $Y$ are independent random variables, then $$U=\frac{1}{X+Y}\ \textrm{and}\  V=\frac{1}{X}-\frac{1}{X+Y}$$
are also independent and distributed according to   $GIG(-p,b,a)$ and $G(p,b)$ laws respectively. 

We recall that the Gamma law $G(p,a)$  with parameters $p,a>0$ is a probability measure that has the density
$$\frac{a^p}{\Gamma(p)}x^{p-1}e^{-ax}\mathbbm{1}_{(0,\infty)}(x)$$
and the Generalized Inverse Gaussian law $GIG(p,a,b)$ with parameters $a,b>0$, $p\in\mathbb{R}$  is a probability  measure that has density $$\frac{(a/b)^{p/2}}{2K_p(2\sqrt{ab})}x^{p-1}e^{-ax-\frac{b}{x}}\mathbbm{1}_{(0,\infty)}(x),$$
where $K_p$ is s  modified Bessel function of the third kind.

Later it was shown in \cite{letac2000independence} that independence of $X$ and $Y$ and independence of $U$ and $V$ characterizes Gamma and GIG laws. In the same paper authors generalized the Matsumoto-Yor property to the framework of real symmetric matrices. Further generalizations of different nature can be found for example in \cite{MassamWeso}, \cite{bao2021characterizations} and \cite{kolodziejek2017matsumoto}.

The analogue of the Matsumoto-Yor property in free probability  was studied in \cite{szpojankowski2017matsumoto}. In this case the property states that if  $\bX, \bY$ are free non-commutative random variables and have free-GIG and Marchenko-Pastur distributions respectively  (with suitably chosen parameters), then 
the random variables $$\bU=(\bX+\bY)^{-1}\ \textrm{and}\ \bV=\bX^{-1}-(\bX+\bY)^{-1}$$ are also free and have free GIG and Marchenko-Pastur distribution.
It was shown in \cite{szpojankowski2017matsumoto} that freenes of $\bX$ na $\bY$ and freenes of $\bU$ and $\bV$  characterizes free-GIG and Marchenko-Pastur laws.

In this paper we  study regression versions of the above characterization,  assuming only constant regressions 
\begin{equation} \left\{\begin{array}{rcl}
\varphi\left(\bV^k\mid\bU\right)&=&m_k\bI,\\
\varphi\left(\bV^l\mid\bU\right)&=&m_l\bI,\\
\end{array} \right.\label{jedenjeden}
\end{equation}
where $k,l\in\mathbb{Z}$ are non zero, $k\neq l$ and $m_k,m_l\in \mathbb{R}$ are some constants. The cases we consider are $(k,l)=(1,2),(1,-1), (-1,-2)$. The case $(k,l)=(-1,1)$ was also studied \cite{szpojankowski2017matsumoto} but the author used a different method based on the moment transform. In classical probability the same case $(k,l)=(-1,1)$ was considered first in \cite{MR1888812} and the remaining cases were considered in \cite{MR2091757}

Our  main tools are subordination of free convolutions  and Boolean cumulants. Subordination is a powerful technique  first used in \cite{biane1998processes} and then enhanced considerably in \cite{voiculescu2000coalgebra}. Roughly speaking for the free additive convolution one has that conditional expectation of the resolvent $(z-\bX-\bY)^{-1}$ onto the algebra generated by $\bX$ is the resolvent of $\bX$ at different point $\oj$, where $\omega_1$ is an analytic selfmap of the upper half-plane $\mathbb{C}^+$. Subordination proved to be very useful in studying properties of free convolutions (see eg. \cite{belinschi2007new, belinschi2006note, belinschi2008lebesgue}) and in random matrix theory (see \cite{belinschi2017outliers}). It was also observed that subordination is useful in regression characterization problems (cf. \cite{ejsmont2017convolution}). For introduction to subordination results we recommend Chapter 2 of \cite{mingofree}.

Powerful as it is, subordination itself does not allow to prove all the results which we are studying here. We take advantage of connections between free probability and Boolean cumulants established recently in \cite{fevrier2020using, lehner2019boolean}. We develop ideas from \cite{lehner2019boolean}, in particular we provide a new expansion of the reciprocal of the additive subordination function in terms of Boolean cumulants  \begin{equation}
\frac{1}{\od}=\sum_{n=0}^\infty\beta_{2n+1}\left(\bRx,\bY,\bRx,\bY,\ldots,\bY,\bRx\right).
\end{equation}
(See Lemma \ref{techlemfor1} for more details.)

One of the implicit results of this paper is a methodological remark, that Boolean cumulants prove to be useful tool when dealing with conditional expectation of expressions involving free random variables. It confirms the observation already noted in \cite{szpojankowski2019conditional} in context of regression versions of the Lukacs property.

The paper is organized as follows: In Sections \ref{background} and \ref{MYproperty} we introduce basic facts from free probability theory and recall the Matsumoto-Yor property in more details. In Section \ref{technicalresults} we derive  some formulas relating subordination functions and Boolean cumulants as well as we relate regression conditions of the form \eqref{jedenjeden} to some equations connecting subordination functions and the Cauchy-Stieltjes transform of $\bX+\bY$. In Section \ref{maintheorems} we state   and prove characterization theorems which are the main results of the paper.

\section{Background and notation}\label{background}
In this section we introduce basic notions and facts from non-commutative probability theory that are needed to understand this paper. We assume we are given a $\cgw$-probability space $\left(\calA,\varphi\right)$ i.e. $\calA$ is a unital $\cgw$-algebra and $\varphi:\calA\to\mathbb{C}$ is positive, tracial and faithful functional (state) such that $\varphi(\bI)=1$ where $\bI$ is  the unit of $\calA$.

Elements of $\calA$ are called (non-commutative) random variables and in this paper are denoted as $\bX, \bY,\bZ$ etc.
\subsection{Freeness and cummulants}
Freeness is one of basic concepts that serves as the analogue  of independence from classic probability theory and was introduced by Voiculescu in \cite{voiculescu1986addition}
\begin{definition}
	We say that unital  subalgebras $\calA_1,...,\calA_n$ of $\calA$ are free if for every choice of centered random variables $\bX_k\in\calA_{i_k}$ (centered, i.e.  $\varphi(\bX_k)=0$), $k=1,2,\ldots,n$, such that $i_1\neq i_2\neq\ldots\neq i_n$ we have 
	$$\varphi\left(\bX_1\cdot\ldots\cdot\bX_n\right)=0.$$
	
	We say that random variables $\bX, \bY\in\calA $ are free if unital subalgebras generated by those elements are free. 
\end{definition}
The definition of freeness can be viewed as a rule for computing joint moments. For example if $\bX, \bY$ are free, then $\varphi(\bX \bY)=\varphi(\bX)\varphi(\bY)$.

For positive integer $n$ let us denote $[n]=\left\{1,2,\ldots,n\right\}$.  
\begin{definition}
	\hfill
\begin{enumerate}
\item A partition $\pi$ of $[n]$ is a set $\pi=\left\{B_1,...,B_k\right\}$ of non-empty and pairwise disjoint subsets of $[n]$ such that $[n]=\bigcup_{i=1}^kB_i$. Elements $B_1,\ldots,B_k$ are called blocks of $\pi$. The set of all partition of $[n]$ is denoted by $\mathcal{P}(n)$.
	\item A partition $\pi\in \mathcal{P}(n)$ is called  an interval partition  if every block $B$ of $\pi$ is of the form $[n]\cap I$ for some interval $I$. The set of all $2^{n-1}$ interval partitions of $[n]$ is denoted by $Int(n)$.
	\item  A partition $\pi\in \mathcal{P}(n)$ is called a non-crossing partition  if for every two blocks $B_1,B_2\in \pi$ and every $i_1,i_2\in B_1$ and $j_1, j_2\in B_2$ such that $i_1<j_1<i_2<j_2$ we have $B_1=B_2$. The set of all non-crossing  partitions of $[n]$ is denoted by $NC(n)$.
\end{enumerate}
\end{definition}
\begin{remark}
	Both sets $Int(n)$ and $NC(n)$ have a lattice structure induced by so-called reversed refinement order. We say that $\pi_1\leq \pi_2$ if every block of partition $\pi_1$ is contained in some block of $\pi_2$.
\end{remark}

\begin{definition}\label{defboolcumu}
	For $n\geq 1$  the Boolean cumulant functional $\beta_n:\mathcal{A}^n\to\mathbb{C}$ and the free cumulant functional $\kappa_n:\mathcal{A}^n\to\mathbb{C}$ are defined recursively by
	$$\forall \bX_1,\ldots,\bX_n\in\mathcal{A}:\varphi\left(\bX_1\cdot\ldots\cdot\bX_n\right)=\sum_{\pi\in Int(n)}\beta_\pi\left(\bX_1,\ldots,\bX_n\right),$$
		$$\forall \bX_1,\ldots,\bX_n\in\mathcal{A}:\varphi\left(\bX_1\cdot\ldots\cdot\bX_n\right)=\sum_{\pi\in NC(n)}\kappa_\pi\left(\bX_1,\ldots,\bX_n\right),$$
	where for $\pi=\left\{B_1,\ldots,B_k\right\}$ 
	$$\beta_\pi\left(\bX_1,\ldots,\bX_n\right)=\prod_{j=1}^{k}\beta_{|B_j|}\left(\bX_i:i\in B_j\right),$$ and
	$$\kappa_\pi\left(\bX_1,\ldots,\bX_n\right)=\prod_{j=1}^{k}\kappa_{|B_j|}\left(\bX_i:i\in B_j\right).$$
\end{definition}
\begin{remark}\label{remarkboolean}
Boolean cumulants can also be defined directly via M\"obius inversion formula as
\begin{equation}\label{booleandef2}
\beta_n\left(\bX_1,\ldots,\bX_n\right)=\sum_{\pi\in Int(n)}(-1)^{|\pi|+1}\varphi_\pi\left(\bX_1,\ldots,\bX_n\right),
\end{equation}
where $|\pi|$ is the number of blocks of $\pi$ and $\varphi_\pi$ is defined in a similar manner to $\beta_\pi$ i.e.
$$\varphi_\pi\left(\bX_1,\ldots,\bX_n\right)=\prod_{j=1}^{k}\varphi\left(\prod_{i\in B_j}\bX_i\right),$$
where $\prod_{i\in B_j}\bX_i=\bX_{k_1}\bX_{k_2}\cdot\ldots\cdot \bX_{k_m}$ if $B_j=\{k_1<k_2<\ldots<k_m\}$. In particular $\beta_1=\varphi$ and $\beta_2(\bX,\bY)=\varphi(\bX\bY)-\varphi(\bX)\varphi(\bY)$.
\end{remark}
We will need two formulas involving Boolean cumulants. They can be found in \cite{fevrier2020using} and \cite{lehner2019boolean} and were used also in \cite{szpojankowski2019conditional}.
\begin{proposition}\label{propbool}
Assume we are given two collections of random variables $\{\bX_1,\bX_2,\ldots,\bX_{n+1}\}$ and $\{\bY_1,\bY_2,\ldots,\bY_{n}\}$ that are free, $n\geq 1$. Then
\begin{multline}\label{boolmain1}
\varphi\left(\bX_1\bY_1\ldots\bX_n\bY_n\right)=\\
=\sum_{k=0}^{n-1}\sum_{0=j_0<j_1<\ldots<j_{k+1}=n}\varphi\left(\bY_{j_1}\ldots\bY_{j_{k+1}}\right)\prod_{l=0}^k\beta_{2(j_{l+1}-j_l)-1}\left(\bX_{j_l+1},\bY_{j_l+1},\ldots, \bY_{j_{l+1}-1},\bX_{j_{l+1}}\right)
\end{multline}
and
\begin{multline*}
\beta_{2n+1}\left(\bX_1,\bY_1,\ldots,\bX_n,\bY_n,\bX_{n+1}\right)=\\
=\sum_{k=2}^{n+1}\sum_{1=j_1<\ldots<j_{k}=n}\beta_k\left(\bX_{j_1},\ldots,\bX_{j_{k}}\right)\prod_{l=1}^{k-1}\beta_{2(j_{l+1}-j_l)-1}\left(\bY_{j_{l}},\bX_{j_l+1},\bY_{j_l+1},\ldots,\bX_{j_{l+1}-1}, \bY_{j_{l+1}-1}\right).
\end{multline*}
\end{proposition}
\begin{remark}
	Formula \eqref{boolmain2} will be used several times in this paper and it will be convenient for the reader if we write it down in the special case when $\{\bX_1,\bX_2,\ldots,\bX_{n+1}\}=\{\bZ_1,\underbrace{\bX,\ldots,\bX}_{n-1},\bZ_{2}\}$ and $\{\bY_1,\bY_2,\ldots,\bY_{n}\}=\{\underbrace{\bY,\bY,\ldots,\bY}_{n}\}$. In this case we have
	\begin{multline*}
	\beta_{2n+1}\left(\bZ_1,\bY,\bX,\ldots,\bX,\bY,\bZ_{2}\right)=\\
	=\sum_{k=2}^{n+1}\sum_{1=j_1<\ldots<j_{k}=n}\beta_k\left(\bZ_1,\underbrace{\bX,\ldots,\bX}_{k-2},\bZ_2\right)\prod_{l=1}^{k-1}\beta_{2(j_{l+1}-j_l)-1}\left(\bY,\bX,\bY,\ldots,\bX,\bY\right).
	\end{multline*}
After simple change of indices this can be written in much simpler form
	\begin{multline}
		\beta_{2n+1}\left(\bZ_1,\bY,\bX,\ldots,\bX,\bY,\bZ_{2}\right)\\
		=\sum_{k=1}^{n}\beta_k\left(\bZ_1,\underbrace{\bX,\ldots,\bX}_{k-1},\bZ_2\right)\sum_{i_1+\ldots+i_{k}=n-k}\prod_{l=1}^{k}\beta_{2i_l+1}\left(\bY,\bX,\bY,\ldots,\bX,\bY\right).\label{boolmain2}
	\end{multline}
\end{remark}
We also recall two  simple facts. 
\begin{proposition}[\cite{fevrier2020using}] \label{cor45}
	Let  $n\geq 2$.  If either $\bX_1=\bI$ or $\bX_n=\bI$, then $\beta_n\left(\bX_1,\ldots,\bX_n\right)=0$.
\end{proposition}
\begin{proposition}\label{prop212}
	For $n\geq 1$
	 $$\beta_{n}\left(\bX_1\cdot\bX_2,\bX_3,\ldots,\bX_{n+1}\right)=\beta_{n+1}\left(\bX_1,\bX_2,\bX_3,\ldots,\bX_{n+1}\right)+\beta_1\left(\bX_1\right)\beta_{n}\left(\bX_2,\bX_3,\ldots,\bX_{n+1}\right).$$
\end{proposition}
The last proposition is a special case of Proposition 2.12 from \cite{fevrier2020using}.
\subsection{Conditional expectation}
Assume that $(\calA,\varphi)$  is a $W^{*}$ probability space, i.e., $\calA$  is a finite von Neumann algebra and
$\varphi$ a faithful, normal, tracial state. If  $\mathcal{B}\subset\calA$ is von Neumann subalgebra, we denote by $\varphi\left(\cdot\mid\mathcal{B}\right)$ the conditional expectation with respect to $\mathcal{B}$. That is $\varphi\left(\cdot\mid\mathcal{B}\right):\calA\to\mathcal{B}$ is  faithful, normal projection such that $\varphi\circ\left[ \varphi\left(\cdot\mid\mathcal{B}\right)\right]=\varphi$. The map $ \varphi\left(\cdot\mid\mathcal{B}\right)$ is a $\mathcal{B}$-module map i.e. $$\varphi\left(\bY_1\bX\bY_2\mid\mathcal{B}\right)=\bY_1\varphi\left(\bX\mid\mathcal{B}\right)\bY_2$$ for all $\bX\in\calA$ and $\bY_1,\bY_2\in\mathcal{B}$.

\subsection{Distribution of a random variable and analytic tools.}
\begin{definition}
The distribution of   self-adjoint random variable $\bX\in\calA$ is a uniquely determined,  compactly supported, probability measure $\mu_\bX$ on the real line such that for all $n\geq 1$
 $$\varphi\left(\bX^n\right)=\int_{\mathbb{R}}x^n\mu_{\bX}(dx).$$
\end{definition}
We list now some analytic tools and their properties that we use in this paper.
\begin{enumerate}
\item The Cauchy-Stieltjes transform of  a compactly supported measure  $\mu$ on the real line is  the map 
$$G_\mu(z)=\int_{\mathbb{R}}\frac{\mu(dx)}{z-x},$$
defined for  $z\in\mathbb{C}\setminus\textrm{supp}(\mu)$.
It is known that the Cauchy-Stieltjes transform is an analytic map $G_\mu:\mathbb{C}^+\to \mathbb{C}^-$.

 If $\bX$ is a self-adjoint random variable we write $G_{\bX}$ for the $G_{\mu_\bX}$. Note that
	$$\gx=\varphi\left((z\bI-\bX)^{-1}\right)=\int_{\mathbb{R}}\frac{\mu_\bX(dx)}{z-x}.$$
	
\item The $r$-transform of $\bX$ is the function $$r_\bX(z)=G_\bX^{-1}(z)-\frac{1}{z},$$ where $G_\bX^{-1}$ is the inverse function of $G_{\bX}$, defined in some neighborhood of $0$. It is known that $r_{\bX}$ is an analytic map and for sufficiently small $z$  the following expansion  holds
	$$r_\bX(z)=\sum_{k=0}^\infty\kappa_{k+1}(\bX,...,\bX)z^k.$$
\item The moment transform of $\bX$ (which is not necessarily self-adjoint) is  defined  for all $z\in\mathbb{C}$ such that $\bI-z\bX$ is invertible  as  $$M_\bX(z)=\varphi\left(z\bX(\bI-z\bX)^{-1}\right).$$
$M_\bX$ is an an analytic function in some  neighborhood of $0$ and one has 
$$M_\bX(z)=	\sum_{k=1}^\infty\varphi(\bX^k)z^k.$$
\item The $\eta$-transform of $\bX$ is defined by $$\eta_\bX(z)=\frac{M_\bX(z)}{M_\bX(z)+1}.$$
 In some neighborhood of $0$ one has 
 $$\eta_\bX(z)=\sum_{k=1}^\infty\beta_{k}(\bX,...,\bX)z^k.$$
 
\end{enumerate}
Each of these transformations uniquely determine moments of a self-adjoint random variable $\bX$ and thus also uniquely determine its distribution.

\subsection{Subordination}
Let $\bX$ and $\bY$ be  free self-adjoint random variables. There is a fundamental relation between  $r$-transforms  of $\bX,\bY$ and $\bX+\bY$, namely
\begin{equation}
r_{\bX+\bY}(z)=r_{\bX}(z)+r_{\bY}(z).\label{rxpy}
\end{equation}
Consequently distributions of $\bY$ and $\bX+\bY$ determine the distribution of $\bX$.

The relation between  Cauchy-Stieltjes transforms of $\bX,\bY$ and $\bX+\bY$ is more complicated and was established by Biane in \cite{biane1998processes}. It involves two  functions $\omega_1,\omega_2$ that can be defined as unique analytic functions $\omega_1,\omega_2:\mathbb{C}^+\to \mathbb{C}^+$ satisfying the following properties: $\textrm{Im}(\omega_k(z))\geq \textrm{Im}(z)$, $\omega_k(iy)/iy\longrightarrow 1$ when $y\to+\infty$, $k=1,2$, and
\begin{equation}\gxy=G_{\bX}(\omega_1(z))=G_{\bY}(\omega_2(z))\label{subordination}
\end{equation} 
Because of the last property $\omega_1,\omega_2$  are called the subordination functions. 

As a consequence of \eqref{rxpy} and \eqref{subordination} the following equality holds for all $z\in\mathbb{C}^+$
\begin{equation}
z=\omega_1(z)+\omega_2(z)-\frac{1}{\gxy}\label{omega12g}
\end{equation}

We also need the following theorems. The first one generalizes  formula \eqref{subordination} in the framework of von Neumann algebras. The second gives interesting series expansion of the subordination function $\oj$ that involves Boolean cumulants.
\begin{proposition}[\cite{biane1998processes}]
	If $\bX$ and  $\bY$ are free self-adjoint random variables, then for all $z\in\mathbb{C}^+$
	\begin{equation}
	\varphi\left((z\bI-\bX-\bY)^{-1}\mid\bX\right)=\left(\oj-\bX\right)^{-1}. \label{resolventeqation}
	\end{equation}
	\end{proposition}
	
\begin{proposition}[\cite{lehner2019boolean}]
	If $\bX$ and $\bY$ are free self-adjoint random variables, then 
	\begin{equation}
	\omega_1(z)=z-\sum_{n=0}^\infty\beta_{2n+1}\left(\bY,\bRx,\bY,\ldots,\bRx,\bY\right)\label{cofz}
	\end{equation}
	in some neighborhood of infinity in $\mathbb{C}^+$.
\end{proposition}

\section{Free Matsumoto-Yor property}\label{MYproperty}
In this section we recall necessary definitions to state the Matsumoto-Yor property in free probability.
\subsection{Free Poisson Distribution} We say that the measure $\nu=\nu(\lambda,\gamma)$ with $\lambda\geq 0, \gamma>0$  is free Poisson or Marchenko-Pastur distribution if
$$\nu=\max\{0,1-\lambda\}\delta_0+\lambda\nu_1,$$
where $\nu_1$ is a probability measure with density
$$\frac{1}{2\pi\gamma x}\sqrt{4\lambda\gamma^2-(x-\gamma(1+\lambda))^2}\ \mathbbm{1}_{\left(\gamma(1-\sqrt{\lambda})^2,\gamma(1+\sqrt{\lambda})^2\right)}(x).$$

The $r$-transform of the free Poisson distribution $\nu(\lambda,\gamma)$ is equal $$r_{\nu(\lambda,\gamma)}(z)=\frac{\lambda\gamma}{1-\gamma z}.$$
\subsection{Free-GIG distribution}
The free Generalized Inverse Gaussian distribution is a probability  measure $\mu=\mu(\lambda, \alpha, \beta)$, with $\alpha,\beta>0, \lambda\in\mathbb{R}$, which is compactly supported on the interval $[a,b]$ and has the density 
$$\frac{d\mu}{dx}=\frac{1}{2\pi}\sqrt{(x-a)(x-b)}\left(\frac{\alpha}{x}+\frac{\beta}{\sqrt{ab}x^2}\right),$$
where $(a,b)$ such that  $0<a<b$ is the unique solution of
\begin{equation*}
\left\{\begin{array}{lc}
1-\lambda+\alpha\sqrt{ab}-\beta\frac{a+b}{ab}&=0,\\
1+\lambda+\frac{\beta}{\sqrt{ab}}-\alpha\frac{a+b}{2}&=0.
\end{array}\right.
\end{equation*}

The Cauchy-Stieltjes transform of the  free-GIG distribution  $\mu=\mu(\lambda, \alpha, \beta)$ is equal
\begin{equation*}
G_\mu(z)=\frac{\alpha z^2-(\lambda-1)z-\beta-(\alpha z+\frac{\beta}{\sqrt{ab}})\sqrt{(z-a)(z-b)}}{2z^2}.
\end{equation*}
See \cite{feral2006limiting} for more details.

It is easy to check that the Cauchy-Stieltjes transform $G=G(z)$ of the  free-GIG distribution $ \mu(\lambda, \alpha, \beta)$ satisfies the following quadratic equation $$	z^2G^2-(\alpha z^2-(\lambda-1)z-\beta)G+\alpha z+\delta=0.$$ where $\delta$ depends on $\alpha, \beta,\lambda$. The following lemma that can be extracted from the proof of (\cite{szpojankowski2017matsumoto}, Theorem 4.1.)  shows the converse of this statement.

\begin{lemma}\label{giglemma}
	Suppose the function $G=G(z)$ satisfies the following  equation
	\begin{equation*}
	z^2G^2-(\alpha z^2-(\lambda-1)z-\beta)G+\alpha z+\delta=0
	\end{equation*}
	i.e. $$G(z)=\frac{\alpha z^2-(\lambda-1)z-\beta\pm\sqrt{(\alpha z^2-(\lambda-1)z-\beta)^2-4z^2(\alpha z+\delta)}}{2z^2},$$
	for some $\alpha, \beta, \delta>0$ and $\lambda\in\mathbb{R}$. If $G$ is the the Cauchy-Stieltjes transform of a positive random variable $\bX$, then $\delta$ is uniquely determined by $\alpha,\beta,\lambda$ and $\bX$ has the  free-GIG distribution $\mu(\lambda,\alpha,\beta)$.
\end{lemma}
\subsection{The free Matsumoto-Yor property.}
The following independence property was observed by Matsumoto and Yor in \cite{matsumoto2001analogue}: If $X\sim GIG(-p,a,b)$ and $Y\sim G(p,a)$ are independent random variables, then $$U=\frac{1}{X+Y}\ \textrm{and}\  V=\frac{1}{X}-\frac{1}{X+Y}$$
 are also independent and distributed   $GIG(-p,b,a)$ and $G(p,b)$ respectively. 
 
 Later it was shown in \cite{letac2000independence} that the  Matsumoto-Yor property characterizes GIG and Gamma laws:
 \begin{theorem}
 	Let $X$ and $Y$ be positive, independent and non-degenerated random variables. If $U=\frac{1}{X+Y}$ and $V=\frac{1}{X}-\frac{1}{X+Y}$ are independent, then  $X\sim GIG(-p,a,b)$ and $Y\sim G(p,a)$. 
 \end{theorem}

The Matsumoto-Yor property in free probability was studied in \cite{szpojankowski2017matsumoto} where the author proved the following theorems: 
\begin{theorem}
	Let $\bX$ and $\bY$ be self-adjoint random variables such that $\bX$ has the free-GIG distribution $\mu(-\lambda,\alpha, \beta)$ and the distribution of $\bY$ is free-Poisson $\nu(\lambda,1/\alpha)$. If $\bX, \bY$ are free, then
	 \begin{equation}\bU=(\bX+\bY)^{-1}\ \textrm{and}\ \bV=\bX^{-1}-(\bX+\bY)^{-1}\label{uuuvvv}\end{equation}
	  are free. Moreover $\bU$ and $\bV$ have   $\mu(-\lambda,\beta,\alpha)$ and $\nu(\lambda,1/\beta)$ distributions respectively. 
\end{theorem}
\begin{theorem}
	Let $\bX$ and $\bY$ be free positive self-adjoint random variables. If $\bU, \bV$ defined as in \eqref{uuuvvv} are free, then  $\bX$ has the  free-GIG distribution $\mu(-\lambda,\alpha, \beta)$ and the distribution of $\bY$ is free-Poisson $\nu(\lambda,1/\alpha)$ for some parameters $\alpha, \beta>0$ and $\lambda\in\mathbb{R}$.
\end{theorem}
The following lemma will also be useful.
\begin{lemma}[\cite{szpojankowski2017matsumoto}, Remark 2.1]\label{szpolemm}
	Let $\bX$ and $\bY$ be self-adjoint random variables such that $\bX$ has the free-GIG distribution $\mu(-\lambda,\alpha, \beta)$ and the distribution of $\bY$ is free-Poisson $\nu(\lambda,1/\alpha)$. Then the distribution of $\bX+\bY$ is free-GIG $\mu(\lambda,\alpha, \beta)$.
\end{lemma}

\section{Analytic interpretation of regression conditions} \label{technicalresults}
In this section we  prove a few auxiliary results that will be useful in the sequel.

\subsection{Subordination vs Boolean cumulants}
\begin{lemma}\label{techlemma1}
Let $\bX$ and $\bY$ be free self-adjoint and compactly supported random variables. Then for $z$ in some neighborhood of infinity in $\mathbb{C}^+$

 \begin{equation}
\sum_{n=0}^\infty\beta_{2n+1}\left(\bRx,\bY,\bRx,\ldots,\bY,\bRx\right)=\frac{1}{\od}.\label{techlemfor1}
	\end{equation}

\end{lemma}
\begin{remark}\label{rem1}
	It's easy to check that for $|z|>||\bX||$ we have $||\bRx||\leq \left(|z|-||\bX||\right)^{-1}$. The  formula \eqref{booleandef2} implies that $|\beta_n(\bX_1,\ldots,\bX_n)|\leq 2^{n-1} ||\bX_1||\cdot\ldots\cdot||\bX_n||$. Hence   for $|z|>||\bX||$ we have
$$|\beta_{2n+1}\left(\bRx,\bY,\bRx,\ldots,\bY,\bRx\right)|\leq2^{2n}\frac{||\bY||^n}{\left(|z|-||\bX||\right)^{n+1}}.$$
This implies that  the series from  Lemma \ref{techlemma1} converges for $|z|>||\bX||+4||\bY||$ and represents a holomorphic function.
\end{remark}
\begin{proof}
To simplify the  notation we will write $\bR$ for the resolvent $\bRx$.

Let us denote the right hand side of \eqref{techlemfor1} by $D(z)$, i.e.  $$D(z)=\sum_{n=0}^\infty\beta_{2n+1}\left(\bR,\bY,\ldots,\bY,\bR\right).$$
It is easy to check the result when $\bY=0$. In this case $\od=\frac{1}{\gx}$ by \eqref{subordination} and the series  consists of one nonzero element $\beta_1(\bR)=\varphi\left(\bRx\right)=\gx$. Thus for the rest of the proof we   assume $\bY\neq 0$. This implies $\omega_1$ is not an identity functions on $\mathbb{C}^{+}$.


	Formula \eqref{boolmain2} implies that for $n\geq 1$  the cumulant  $\beta_{2n+1}(\bR,\bY,\bR,\ldots,\bY,\bR)$ is equal to 
	$$\sum_{k=1}^{n}\beta_{k+1}(\bR,\bR,\ldots,\bR)\sum_{i_1+\ldots+ i_k=n-k}\prod_{l=1}^k\beta_{2i_{l}+1}(\bY,\bR,\ldots,\bR,\bY)$$
After changing the order of summation one can see that
\begin{equation*}
\begin{split}
D(z)&=\beta_1(\bR)+\sum_{n=1}^\infty\sum_{k=1}^{n}\beta_{k+1}(\bR,\ldots,\bR)\sum_{i_1+\ldots+ i_k=n-k}\prod_{l=1}^k\beta_{2i_{l}+1}(\bY,\bR,\ldots,\bR,\bY)\\
&=\beta_1(\bR)+\sum_{k=1}^{\infty}\beta_{k+1}(\bR,\ldots,\bR)\sum_{n=k}^{\infty}\sum_{i_1+\ldots+ i_k=n-k}\prod_{l=1}^k\beta_{2i_{l}+1}(\bY,\bR,\ldots,\bR,\bY)\\
&=\beta_1(\bR)+\sum_{k=1}^{\infty}\beta_{k+1}(\bR,\ldots,\bR)C(z)^k,
\end{split}
\end{equation*}
	where $$C(z)=\sum_{n=0}^\infty\beta_{2n+1}(\bY,\bR,\ldots,\bR,\bY).$$
	Thus, in view of \eqref{cofz} we can write $C(z)=z-\omega_1(z).$
	
	If $C(z)\neq 0$ one  can write 
	$$D(z)=\frac{\eta_{\bR}\left(C(z)\right)}{C(z)}=\frac{M_{\bR}\left(C(z)\right)}{C(z)\left[M_{\bR}\left(C(z)\right)+1\right]}.$$
	Easy algebraic manipulations and forumula \eqref{subordination} show that
	\begin{multline*}
	M_{\bR}\left(C(z)\right)=\varphi\left(C(z)\bR(\bI-C(z)\bR)^{-1}\right)=C(z)\varphi\left((\bR^{-1}-C(z)\bI)^{-1}\right)\\
	=C(z)\varphi\left((\oj\bI-\bX)^{-1}\right)=C(z)\gxy.
	\end{multline*}
	Thus $$D(z)=\frac{\gxy}{1+(z-\oj)\gxy}=\frac{\gxy}{1+\left(\od-\gxy^{-1}\right)\gxy}=\frac{1}{\od},$$
	where we used formula \eqref{omega12g}. This proves the lemma for all sufficiently large $z\in \mathbb{C}^+$ such that $C(z)\neq 0$ but since $C(z)$ is a nonzero analytic function this last assumption can be dropped.

\end{proof}

\begin{lemma}\label{techlemma2}
	Let $\bX$ and $\bY$ be free self adjoint and compactly supported random variables such that $\bY$ is invertible. Then for all sufficiently large $z\in\mathbb{C}^+$
	\begin{equation}\label{techlemfor2} A(z):=\sum_{n=0}^\infty\beta_{2n+1}(\bY^{-1},\underbrace{\bRx,\bY,\ldots,\bRx,\bY}_{2n})=\frac{1}{\omega_2(z)}+\frac{\varphi(\bY^{-1})}{\omega_2(z)\gxy}.
	\end{equation}
	\begin{equation}\label{techlemfor3}
B(z):=\sum_{n=1}^\infty\beta_{2n+1}(\bY^{-1},\bRx,\bY,\ldots,\bY,\bRx,\bY^{-1})=\frac{\varphi(\bY^{-2})-\varphi(\bY^{-1}) A(z)}{\omega_2(z)}.
\end{equation}
\end{lemma}
\begin{proof} 
	We will write $\bR$ for $\bRx$ to simplify the notation. Hence
 $$A(z)=\sum_{n=0}^\infty\beta_{2n+1}(\bY^{-1},\bR,\bY,\ldots,\bR,\bY).$$
	Applying formula \eqref{boolmain2} we can see  that for $n\geq 1$ the Boolean cumulant  $\beta_{2n+1}(\bY^{-1},\bR,\ldots,\bR,\bY)$ is equal
		$$\sum_{k=1}^{n}\beta_{k+1}(\bY^{-1},\bY,\ldots,\bY)\sum_{i_1+\ldots+ i_k=n-k}\prod_{l=1}^k\beta_{2i_l+1}(\bR,\bY,\ldots,\bY,\bR).$$
		The same argument as in the previous lemma shows that for sufficiently large $z\in\mathbb{C}^+$
	\begin{equation}
	A(z)=\beta_1(\bY^{-1})+\sum_{k=1}^{\infty}\beta_{k+1}(\bY^{-1},\bY,\ldots,\bY)D(z)^k\label{wzrownaaz}
\end{equation}
	where  $D(z)=\sum_{n=0}^\infty\beta_{2n+1}(\bR,\bY,\bR,\ldots,\bY,\bR)=\frac{1}{\od}$ by Lemma \ref{techlemma1} 
	
	From Propositions \ref{cor45},  \ref{prop212} and Remark \ref{remarkboolean} we can deduce that
	$$\beta_{k+1}(\bY^{-1},\bY,\ldots,\bY)=\left\{\begin{array}{lc}
	\varphi(\bY^{-1}),&k=0\\
	1-\varphi(\bY^{-1})\beta_1(\bY),&k=1\\
	-\varphi(\bY^{-1})\beta_{k}(\bY,\bY,\ldots,\bY),& k\geq 2
	\end{array} \right..$$
	Hence $$A(z)=\varphi(\bY^{-1})+	(1-\varphi(\bY^{-1})\beta_1(\bY))D(z)-\varphi(\bY^{-1})\sum_{k=2}^{\infty}\beta_{k}(\bY,\ldots,\bY)D(z)^k$$
	$$=D(z)+\varphi(\bY^{-1})\left(1-\eta_\bY(D(z))\right).$$
	
	Now its easy to check that $1-\eta_\bY(z)=\frac{1}{\varphi\left((\bI-z\bX)^{-1}\right)}=\frac{z}{G_\bX\left(\frac{1}{z}\right)}$.
	Thus $$A(z)=\frac{1}{\od}+\frac{\varphi(\bY^{-1})}{\od G_\bY(\od)}=\frac{1}{\omega_2(z)}+\frac{\varphi(\bY^{-1})}{\omega_2(z)\gxy}.$$
 Now we can prove formula \eqref{techlemfor3}. Using formula \eqref{boolmain2} one more time one can see that

$$B(z)=\sum_{k=1}^{\infty}\beta_{k+1}(\bY^{-1},\underbrace{\bY,\ldots,\bY}_{k-1},\bY^{-1})D(z)^{k}.$$
It follows from Propositions \ref{cor45}, \ref{prop212} and the fact that Boolean cumulants are invariant under reflection (i.e. $\beta_n(\bX_1,\bX_2,\ldots,\bX_n)=\beta_n(\bX_n,\ldots,\bX_2,\bX_1)$) that
$$\beta_{k+2}(\bY^{-1},\bY,\ldots,\bY,\bY^{-1})=\left\{\begin{array}{lc}
\varphi(\bY^{-2})-\varphi(\bY^{-1})^2,&k=0\\
-\varphi(\bY^{-1})\beta_{k+1}(\bY^{-1},\bY,\ldots,\bY,\bY),& k\geq 1
\end{array} \right..$$
Consequently
 $$B(z)=\frac{\varphi(\bY^{-2})-\varphi(\bY^{-1})^2}{\od}-\frac{\varphi(\bY^{-1})}{\od}\left(\sum_{k=0}^{\infty}\beta_{k+1}(\bY^{-1},\bY,\ldots,\bY)D(z)^k-\varphi(\bY^{-1})\right).$$
The series in the above  expression is exactly $A(z)$ (formula \eqref{wzrownaaz}.) This ends the proof of the lemma.
\end{proof}
\begin{remark}
	Consider (formal) power series 
	$$
	\eta^f_\bY(z)=\sum_{k\ge 0}\,\beta_{k+1}(f(\bY),\underbrace{\bY,\ldots,\bY}_{k})\,z^k
	$$
	and
	$$
	\eta^{f,g}_\bY(z)=\sum_{k\ge 0}\,\beta_{k+2}(f(\bY),\underbrace{\bY,\ldots,\bY}_{k},g(\bY))\,z^k
	$$
	for $f,g:\mathcal A\to\mathcal A$, which seem to be important in relations between subordination and Boolean cumulants. In (\cite{szpojankowski2019conditional}, Proposition 3.4) a general and rather complicated formula expressing $\eta^f_\bY$ and $\eta^{f,g}_Y$ in terms of $\eta_Y$ was proved for $f$ and $g$ being analytic functions in the unit disc. Consequently, in a special case of  $0\le \bY<\bI$ and $f(\bY)=g(\bY)=\psi(\bY)=\bY(1-\bY)^{-1}$ explicit expressions were derived there (see the proof of (\cite{szpojankowski2019conditional}, Proposition 3.7)
	$$
	\eta_\bY^{\psi}(z)=\frac{\eta_\bY(z)-\eta_\bY(1)}{z-1}\,\varphi((1-\bY)^{-1})
	$$
	and
	$$
	\eta_\bY^{\psi,\psi}(z)=\frac{\eta_\bY(z)-\eta_\bY(1)-(z-1)\eta'_\bY(1)}{(z-1)^2}\varphi^2((1-\bY)^{-1}).
	$$
	
	It is interesting to note that in the proof of Lemma 4.3 we actually derived formulas for $\eta_\bY^h$ and $\eta_\bY^{h,h}$ for $h(\bY)=\bY^{-1}$ (which clearly is not analytic in the unit disc). Namely, the formula for $A(z)$ gives
	$$
	\eta_\bY^h(z)=z+(1-\eta_\bY(z))\varphi(\bY^{-1})
	$$
	and the formula for $B(z)$ yields
	$$
	\eta_\bY^{h,h}(z)=\varphi(\bY^{-2})-z\varphi(\bY^{-1})+\varphi^2(\bY^{-1})(\eta_\bY(z)-1).
	$$
\end{remark}	
	
\subsection{Constant regressions and their implications}\label{derequ}
From now on we assume we are given $W^*$ probability space $\left(\calA,\varphi\right)$. We also assume $\bX,\bY\in\calA$ are free, self-adjoint and positive random variables and $\bU, \bV$ are defined as follows.
 $$\bU=(\bX+\bY)^{-1},\ \ \ \ \bV=\bX^{-1}-(\bX+\bY)^{-1}.$$
 In this subsection we show that the condition of constant regression $$\varphi(\bV^{k}\mid\bU)=m_k\bI$$ in each of considered cases i.e. for $k\in\left\{-2,-1,1,2\right\}$ implies certain equation that connects the Cauchy-Stieltjes transform $G_{\bX+\bY}$ as well as subordination functions $\omega_1$ and $\omega_2$. We will consider each case separately. The most challenging was the case  $k=-2$. In all other cases subordination was enough to to get the result. In the case  $k=-2$ we additionally have to rely on Boolean cumulants.

To simplify the notation we will denote $\bT=\bX+\bY=\bU^{-1}$. Note that $\varphi\left(\cdot\mid\bT\right)=\varphi\left(\cdot\mid\bU\right)$ since we assumed $\bX, \bY$ are positive.

 We  also introduce the rational functions $q_1(t,z)$ and $q_2(t,z)$  in variable $t$ and their partial fraction decompositions
$$q_1(t,z)=\frac{1}{t(z-t)}=\frac{1}{zt}+\frac{1}{z(z-t)}$$
and
$$ q_2(t,z)=\frac{1}{t^2(z-t)}=\frac{1}{zt^2}+\frac{1}{z^2t}+\frac{1}{z^2(z-t)}$$
where $z\in \mathbb{C}^+$.
\begin{lemma}\label{lemmap1}
	Let us assume \begin{equation}\varphi\left(\bV\mid \bU\right)=c\bI, \label{reg1} \end{equation}
for some constants $c$. Then  
	\begin{equation}
\frac{1}{\omega_1(z)}\left(\varphi(\bU)+c+\gxy\right)=\left(c+\frac{1}{z}\right)\gxy+\frac{\varphi(\bU)}{z}\label{MYeq1}
	\end{equation}
		for all $z\in\mathbb{C}^{+}$.

\end{lemma}
\begin{proof}
We start by rewriting  \eqref{reg1} as
\begin{equation}
\varphi\left(\bX^{-1}\mid \bT\right)=c\bI+\bT^{-1}.\label{phixneg1}
\end{equation}
When we multiply both sides from the right by  $(z\bI-\bT)^{-1}$ and apply  $\varphi$ we get
\begin{equation}
\varphi\left(\bX^{-1}(z\bI-\bT)^{-1}\right)=c\varphi\left((z\bI-\bT)^{-1}\right)+\varphi\left(\bT^{-1}(z\bI-\bT)^{-1}\right).\label{row1}
\end{equation}
Since $\bT^{-1}(z\bI-\bT)^{-1}=q_1(\bT,z)=\frac{1}{z}\bT^{-1}+\frac{1}{z}(z\bI-\bT)^{-1}$
the right hand side of \eqref{row1} becomes

$$\left(c+\frac{1}{z}\right)\varphi\left((z\bI-\bT)^{-1}\right)+\frac{1}{z}\varphi\left(\bT^{-1}\right)=\left(c+\frac{1}{z}\right)\gxy+\frac{\varphi(\bU)}{z}.$$
Now we deal with the left hand side of \eqref{row1}. By conditioning on $\bX$ we see that
\begin{equation}\varphi\left(\bX^{-1}(z\bI-\bT)^{-1}\right)=\varphi\left(\bX^{-1}\varphi\left(\left(z\bI-\bX-\bY\right)^{-1}\mid\bX\right)\right).\label{wzor1}
\end{equation}
Formula \eqref{resolventeqation} implies  this is equal to
\begin{equation*}
\begin{split}
\varphi\left(\bX^{-1}\left(\omega_1(z)-\bX\right)^{-1}\right)&=\varphi\left(\frac{1}{\omega_1(z)}\bX^{-1}+\frac{1}{\omega_1(z)}\left(\omega_1(z)-\bX\right)^{-1}\right)\\
&=\frac{1}{\omega_{1}(z)}\varphi(\bX^{-1})+\frac{1}{\omega_1(z)}G_{\bX}(\omega_1(z))\\
&=\frac{1}{\omega_{1}(z)}\left(\varphi(\bU)+c+\gxy\right).
\end{split}
\end{equation*}
by subordination property \eqref{subordination} and by the equality  $\varphi(\bX^{-1})=\varphi(\bU)+c$ that follows from \eqref{phixneg1}.

\end{proof}
\begin{lemma}\label{lemmam1}
	Let us assume 	that \begin{equation}\varphi\left(\bV^{-1}\mid \bU\right)=d\bI, \label{conditionD}\end{equation}
for some constant $d$. Then  
	\begin{equation}
	\frac{1}{\od}\left(\varphi(\bY^{-1})+\gxy\right)=\frac{d\varphi(\bU)}{z^2}+\frac{\varphi(\bY^{-1})}{z}+\left(\frac{d}{z^2}+\frac{1}{z}\right)\gxy\label{MYeq2}
	\end{equation}
	for all $z\in\mathbb{C}^{+}$.

	\end{lemma}
\begin{proof}

We start by noting that $\bT\bV=(\bX+\bY)\left(\bX^{-1}-\left(\bX+\bY\right)^{-1}\right)=\bY\bX^{-1}$. This implies $\bV^{-1}\bT^{-1}=\bX\bY^{-1}$ so if we multiply both sides   of \eqref{conditionD} from the right by $\bT^{-1}$ we get
\begin{equation*}
d\bT^{-1}=\varphi\left(\bX\bY^{-1}\mid \bT\right)=\varphi\left((\bT-\bY)\bY^{-1}\mid \bT\right)=\bT\varphi(\bY^{-1}\mid\bT)-\bI.
\end{equation*}
Hence \begin{equation}
\varphi(\bY^{-1}\mid\bT)=\bT^{-1}+d\bT^{-2}.\label{drugimoment}
\end{equation}
After multiplying both sides by $(z\bI-\bT)^{-1}$ and applying $\varphi$ one can see that
\begin{equation}
\varphi\left(\bY^{-1}(z\bI-\bT)^{-1}\right)=\varphi\left(\bT^{-1}(z\bI-\bT)^{-1}\right)+d\varphi\left(\bT^{-2}(z\bI-\bT)^{-1}\right).\label{row2}
\end{equation}
The left hand side of the above expression is the same as the left hand  side of \eqref{wzor1} with $\bX$ and $\bY$ swapped so by analogy we get
$$\varphi\left(\bY^{-1}(z\bI-\bT)^{-1}\right)=	\frac{1}{\omega_2(z)}\left(\varphi(\bY^{-1})+\gxy\right).$$
Now we calculate the right hand side of \eqref{row2}. 
$$\varphi\left(\bT^{-1}(z\bI-\bT)^{-1}\right)=\varphi\left(q_1(\bT,z)\right)=\frac{\varphi(\bT^{-1})}{z}+\frac{1}{z}\gxy.$$
Similarly
$$\varphi\left(\bT^{-2}(z\bI-\bT)^{-1}\right)=\varphi(q_2(\bT,z))=\frac{1}{z}\varphi(\bT^{-2})+\frac{1}{z^2}\varphi(\bT^{-1})+\frac{1}{z^2}\gxy.$$
Consequently  the left hand side of \eqref{row2} is equal
$$\frac{\varphi(\bU)}{z}+\frac{1}{z}\gxy+\frac{d}{z}\varphi(\bT^{-2})+\frac{d\varphi(\bU)}{z^2}+\frac{d}{z^2}\gxy.$$
This is exactly the right hand side of \eqref{MYeq2} as \eqref{drugimoment} implies that $$\varphi(\bY^{-1})=\varphi(\bU)+d\varphi(\bT^{-2}).$$
\end{proof}
\begin{lemma}\label{lemmap2}
	Let us assume that \begin{equation}\varphi\left(\bV^2\mid \bU\right)=b\bI, \label{reg4} \end{equation}
	for some constant $b$. Then  
	\begin{multline}
\frac{\varphi(\bX^{-2})}{\oj}+\frac{1}{\oj}\left(\varphi(\bX^{-1})+\gxy\right)\left(\frac{1}{\oj}-\frac{2}{z}\right)\\
=\frac{\varphi(\bX^{-2})-b}{z}-\frac{\varphi(\bU)}{z^2}+\left(b-\frac{1}{z^2}\right)\gxy
	\end{multline}
	for all $z\in\mathbb{C}^+$.
\end{lemma}
\begin{proof}
We start by expanding
 $\bV^2=\left(\bX^{-1}-\bT^{-1}\right)^2=\bX^{-2}-\bX^{-1}\bT^{-1}-\bT^{-1}\bX^{-1}+\bT^{-2}$. The condition \eqref{reg4} implies now
\begin{equation*}
\varphi\left(\bX^{-2}\mid\bT\right)-\bT^{-1}\varphi\left(\bX^{-1}\mid\bT\right)-\varphi\left(\bX^{-1}\mid\bT\right)\bT^{-1}+\bT^{-2}=b\bI
\end{equation*}
or equivalently
\begin{equation*}
\varphi\left(\bX^{-2}\mid\bT\right)-2\varphi\left(\bX^{-1}\mid\bT\right)\bT^{-1}=b\bI-\bT^{-2}.\label{l3eq}
\end{equation*}
If we multiply both sides from the right  by the resolvent $\left(z\bI-\bT\right)^{-1}$ and apply $\varphi$ we get
\begin{multline}\label{wzor2}
\varphi\left(\bX^{-2}(z\bI-\bT)^{-1}\right)-2\varphi\left(\bX^{-1}\bT^{-1}(z\bI-\bT)^{-1}\right)\\
=b\varphi\left((z\bI-\bT)^{-1}\right)-\varphi\left(\bT^{-2}(z\bI-\bT)^{-1}\right).
\end{multline}
The right hand side of \eqref{wzor2} is equal to 
$$b\gxy-\varphi\left(q_2(\bT,z)\right)
=b\gxy-\frac{\varphi(\bU^2)}{z}-\frac{\varphi(\bU)}{z^2}-\frac{1}{z^2}\gxy.$$
Now we evaluate the left hand side of \eqref{wzor2}. Note that
\begin{equation}
\varphi\left(\bX^{-2}(z\bI-\bT)^{-1}\right)=\varphi\left(\bX^{-2}\varphi\left(z\bI-\bX-\bY)^{-1}\mid\bX\right)\right)
\end{equation}
Using \eqref{resolventeqation} we see that the last expression is equal
\begin{equation}
\begin{split}
\varphi\left(\bX^{-2}(w_1(z)-\bX)^{-1}\right)&=\varphi\left(q_2(\bX,\omega_1(z))\right)\\
&=\frac{1}{\omega_1(z)}\varphi(\bX^{-2})+\frac{1}{\omega_1^2(z)}\varphi(\bX^{-1})+\frac{1}{\omega_1^2(z)}\varphi\left((\omega_1(z)-\bX)^{-1}\right)\\
&=\frac{1}{\omega_1(z)}\varphi(\bX^{-2})+\frac{1}{\omega_1^2(z)}\left(\varphi(\bX^{-1})+\gxy\right)
\end{split}
\end{equation}
Similarly we have 
\begin{equation}
\begin{split}
\varphi\left(\bX^{-1}\bT^{-1}(z\bI-\bT)^{-1}\right)&=\varphi\left(\bX^{-1}\cdot q_1(\bT,z)\right)\\
&=\frac{1}{z}\varphi\left(\bX^{-1}\bT^{-1}\right)+\frac{1}{z}\varphi\left(\bX^{-1}(z\bI-\bT)^{-1}\right)\\
&=\frac{1}{z}\varphi\left(\bX^{-1}\bT^{-1}\right)+\frac{1}{z\oj}\left(\varphi(\bX^{-1})+\gxy\right).
\end{split}
\end{equation}
The result follows now by simple algebra and on noting that \eqref{l3eq} yields
$$2\varphi\left(\bX^{-1}\bT^{-1}\right)=\varphi\left(\bX^{-2}\right)+\varphi\left(\bU^2\right)-b.$$

\end{proof}
\begin{lemma}\label{lemmam2}
	Let us assume that \begin{equation}\varphi\left(\bV^{-2}\mid \bU\right)=h\bI, \label{reg5} \end{equation}
	for some constant $h$.
Then  
	\begin{multline}
	\varphi(\bX^2)B(z)+A(z)^2\left(\omega_1^2(z)\gxy-\omega_1(z)-\varphi(\bX)\right)=\\
	=h\left(\frac{\varphi(\bU^2)}{z}+\frac{\varphi(\bU)}{z^2}+\frac{1}{z^2}\gxy\right)\label{MYeq4}
	\end{multline}
	for all $z\in\mathbb{C}^+$,
	where
	$$A(z)=\frac{1}{\omega_2(z)}+\frac{\varphi(\bY^{-1})}{\omega_2(z)\gxy},$$
	$$B(z)=\frac{\varphi(\bY^{-2})-\varphi(\bY^{-1}) A(z)}{\omega_2(z)}=\frac{\varphi(\bY^{-2})}{\omega_2(z)}-\frac{\varphi(\bY^{-1})}{\omega_2^2(z)}-\frac{\varphi(\bY^{-1})^2}{\omega_2^2(z)\gxy}.$$

\end{lemma}
\begin{proof}
As in the previous lemmas let us denote $\bT=\bX+\bY=\bU^{-1}$. Since $\bT\bV=\bY\bX^{-1}$ and $\bV\bT=\bX^{-1}\bY$ we see that $\bT^{-1}\bV^{-2}\bT^{-1}=\bY^{-1}\bX^{2}\bY^{-1}$. Hence, after multiplying \eqref{reg5} by $\bT^{-1}$ from both sides, we get
\begin{equation}
\varphi\left(\bY^{-1}\bX^{2}\bY^{-1}\mid\bU\right)=h\bT^{-2}.\label{lm2eq1}
\end{equation}
Let us multiply \eqref{lm2eq1} by $(z\bI-\bT)^{-1}$ from the right and apply $\varphi$ to obtain
\begin{equation*}
\varphi\left(\bY^{-1}\bX^{2}\bY^{-1}(z\bI-\bT)^{-1}\right)=h\varphi\left(\bT^{-2}(z\bI-\bT)^{-1}\right).
\end{equation*}
The right hand is equal to $$h\varphi(q_2(\bT,z))=h\left(\frac{1}{z}\varphi(\bT^{-2})+\frac{1}{z^2}\varphi(\bT^{-1})+\frac{1}{z^2}\gxy\right),$$ which is exactly the righ hand side of \eqref{MYeq4}.
\\To calculate the left hand side we observe  that 
 $$(z\bI-\bT)^{-1}=(z\bI-\bX-\bY)^{-1}=\left(\bI-\left(z\bI-\bX\right)^{-1}\bY\right)^{-1}\left(z\bI-\bX\right)^{-1}.$$
 When $|z|>||\bX||+||\bY||$ we can write
 $$(z\bI-\bT)^{-1}=\sum_{n=1}^{\infty}\left[(z\bI-\bX)^{-1}\bY\right]^{n-1}(z\bI-\bX)^{-1}.$$

Using the above expansion and traciality of $\varphi$ we see that

\begin{multline}
\varphi\left(\bY^{-1}\bX^{2}\bY^{-1}(z\bI-\bT)^{-1}\right)=\sum_{n=1}^{\infty}\varphi\left(\bY^{-1}\bX^{2}\bY^{-1}\left[\bR\bY\right]^{n-1}\bR\right)\\
=\sum_{n=1}^{\infty}\varphi\left(\bY^{-1}\left[\bR\bY\right]^{n-1}\bR\bY^{-1}\bX^2\right),
\end{multline}
 where $\bR=(z\bI-\bX)^{-1}$.
 
Formula \eqref{boolmain1} for the collections $\{\underbrace{\bY^{-1},\bY,\ldots,\bY,\bY^{-1}}_{n+1}\}$ and $\{\underbrace{\bR,\bR,\ldots,\bR,\bX^2}_{n+1}\}$ implies that
\begin{multline*}
\varphi\left(\bY^{-1}\left[\bR\bY\right]^{n-1}\bR\bY^{-1}\bX^2\right)=\varphi(\bX^2)\beta_{2n+1}(\bY^{-1},\bR,\bY,\ldots,\bY,\bR,\bY^{-1})+\\
\sum_{k=1}^n\varphi\left(\bR^k\bX^2\right)\sum_{i_1+\ldots+i_{k+1}=n-k}\beta_{2i_1+1}(\bY^{-1},\bR,\ldots,\bR,\bY)\cdot\ldots \cdot\beta_{2i_{k+1}+1}(\bY,\bR,\ldots,\bR,\bY^{-1}).
\end{multline*}

Hence we get
\begin{multline*}
\varphi\left(\bY^{-1}\bX^{2}\bY^{-1}(z\bI-\bT)^{-1}\right)=\varphi(\bX^2)\sum_{n=1}^{\infty}\beta_{2n+1}(\bY^{-1},\bR,\bY,\ldots,\bY,\bR,\bY^{-1})+\\
+\varphi\left(\sum_{n=1}^\infty\sum_{k=1}^n\left(\sum_{i_1+\ldots+i_{k+1}=n-k}\beta_{2i_1+1}(\bY^{-1},\bR,\ldots,\bR,\bY)\cdot\ldots \cdot\beta_{2i_{k+1}+1}(\bY,\bR,\ldots,\bR,\bY^{-1})\right)\bR^k\bX^2\right).
\end{multline*}
The inner expression is equal to 
\begin{multline*}
\bR\bX^2\sum_{k=1}^\infty\sum_{n=k}^\infty\left(\sum_{i_1+\ldots+i_{k+1}=n-k}\beta_{2i_1+1}(\bY^{-1},\bR,\ldots,\bR,\bY)\cdot\ldots \cdot\beta_{2i_{k+1}+1}(\bY,\bR,\ldots,\bR,\bY^{-1})\right)\bR^{k-1}\\
=A(z)\tilde{A}(z)\bR\bX^2\sum_{k=1}^\infty\left[C(z)\bR\right]^{k-1}=A(z)\tilde{A}(z)\bR\bX^2\left(\bI-C(z)\bR\right)^{-1},
\end{multline*}
where we denoted
$$A(z)=\sum_{n=0}^\infty\beta_{2n+1}(\bY^{-1},\bR,\ldots,\bR,\bY),$$
$$\tilde{A}(z)=\sum_{n=0}^\infty\beta_{2n+1}(\bY,\bR,\ldots,\bR,\bY^{-1})$$
and $$C(z)=\sum_{n=0}^\infty\beta_{2n+1}(\bY,\bR,\ldots,\bR,\bY).$$
Let us additionally denote $$B(z)=\sum_{n=1}^\infty\beta_{2n+1}(\bY^{-1},\bR,\bY,\ldots,\bY,\bR,\bY^{-1}).$$
(Note that each of the above series is convergent for large $z$ by  argument from  Remark \ref{rem1}.)

So far we  established that
\begin{equation*}
\varphi\left(\bY^{-1}\bX^{2}\bY^{-1}(z\bI-\bT)^{-1}\right)=B(z)\varphi\left(\bX^2\right)+A(z)\tilde{A}(z)\varphi\left(\bR\bX^2\left(\bI-C(z)\bR\right)^{-1}\right).
\end{equation*}
Since Boolean cumulants are invariant with respect to reflection we see that $\tilde{A}(z)=A(z)$. Moreover Lemma \ref{techlemfor2} shows that $A(z)$ and $B(z)$ have desired forms.

The remaining objective is to calculate $\varphi\left(\bR\bX^2\left(\bI-C(z)\bR\right)^{-1}\right)$. From  formula \eqref{cofz} we know that $C(z)=z-\oj$. Consequently

 \begin{multline*}
\bR\bX^2\left(\bI-C(z)\bR\right)^{-1}=\bX^2\left(\bR^{-1}-C(z)\bI\right)^{-1}=\bX^2\left(\omega_1(z)\bI-\bX\right)^{-1}\\
=-\omega_1(z)\bI-\bX+\omega_1^2(z)(\omega_1(z)\bI-\bX)^{-1}.
 \end{multline*}
Thus referring again to \eqref{subordination} we get
$$\varphi\left(\bR\bX^2\left(\bI-C(z)\bR\right)^{-1}\right)=\omega_1^2(z)\gxy-\omega_1(z)-\varphi(\bX).$$
This proves the result for all large enough in $z\in\mathbb{C}^+$. Since both sides of \eqref{MYeq4} are analytic functions \eqref{MYeq4} holds for all $z\in\mathbb{C}^+$.

\end{proof}

\section{Characterization theorems}\label{maintheorems}
The aim of this section is to prove regression characterizations  which are our main results. We will deal with  each  case $(k,l)=(1,-1),(1,2),(-1,-2)$ as given in \eqref{jedenjeden} separately. The proof of each case will be broken into the series of lemmas and corollaries.
 We will start with the following case.
\subsection{The case $(k,l)=(1,-1)$}
\begin{theorem}\label{th1}
Let $\bX$ and  $ \bY$ be  free, positive, self-adjoint random variables.
 Let us  define $\bU=\left(\bX+\bY\right)^{-1}$ and  $\bV=\bX^{-1}-\left(\bX+\bY\right)^{-1}$. If the following conditions are satisfied 
 \begin{equation*}\varphi\left(\bV\mid \bU\right)=c\bI,  \end{equation*}
  \begin{equation*}\varphi\left(\bV^{-1}\mid \bU\right)=d\bI,\end{equation*}
  for some constants $c$ and $d$, then $cd>1$ and $\bX$ has the free-GIG distribution   $\mu\left(-\frac{cd}{cd-1},\frac{\gamma}{cd-1},\frac{d}{cd-1}\right)$ and $\bY$ has the free Poisson distribution $\nu\left(\frac{cd}{cd-1},\frac{cd-1}{\gamma}\right)$, where $\gamma$ is some  positive constant.
\end{theorem}
\begin{proof}
Under the  assumptions of   Theorem \ref{th1}, Lemmas \ref{lemmap1}, \ref{lemmam1}  and equation \eqref{omega12g} imply the following system of equations
\begin{equation} \label{ch1sys1} \left\{\begin{array}{rcl}
\frac{1}{\omega_1(z)}\left(\beta+c+\gxy\right)&=&\left(c+\frac{1}{z}\right)\gxy+\frac{\beta}{z}\\
	\frac{1}{\omega_2(z)}\left(\gamma+\gxy\right)&=&\frac{d\beta}{z^2}+\frac{\gamma}{z}+\left(\frac{d}{z^2}+\frac{1}{z}\right)\gxy\\
	z&=&\omega_1(z)+\omega_2(z)-\frac{1}{\gxy}
\end{array} \right.,
\end{equation}
where $\beta=\varphi(\bU)$ and $\gamma=\varphi(\bY^{-1})$ are  positive constants.

 Moreover  $cd=\varphi(\bV)\varphi(\bV^{-1})>1$ by the Cauchy-Schwarz inequality.

 \begin{lemma}
 	The r-transform of $\bY$ is equal
 	$$r_{\bY}(z)=\frac{cd}{\gamma-(cd-1)z},$$
 	hence  $\bY$ has the free Poisson distribution $\nu\left(\frac{cd}{cd-1},\frac{cd-1}{\gamma}\right)$.
\end{lemma}
\begin{proof}

	From the first equation of \eqref{ch1sys1} we see that
	\begin{equation*}
	\left(\beta+\gxy\right)\left(\frac{1}{\oj}-\frac{1}{z}\right)=c\left(\gxy-\frac{1}{\oj}\right).
	\end{equation*}
Note that $\oj$ is not an identity function because otherwise  formula \eqref{subordination} would imply that  $\bY$ has a  degenerate distribution. This allows us to write
\begin{equation}
\beta+\gxy=\frac{c\left(\gxy-\frac{1}{\oj}\right)}{\frac{1}{\oj}-\frac{1}{z}}.\label{bplusG}
\end{equation}
The second equation of \eqref{ch1sys1} can be written in the following form
\begin{equation*}
\left(\gamma+\gxy\right)\left(\frac{1}{\od}-\frac{1}{z}\right)=d\frac{\beta+\gxy}{z^2}.
\end{equation*}
Using \eqref{bplusG} and formula \eqref{omega12g} we see that the right hand side of the above equation is equal
\begin{equation*}
\begin{split}
\frac{cd}{z^2}\frac{\left(G-\frac{1}{\oj}\right)}{\frac{1}{\oj}-\frac{1}{z}}&=
\frac{cdG}{z}\frac{\oj-\frac{1}{G}}{z-\oj}\\
&=\frac{cdG}{z}\frac{z-\od}{\od-\frac{1}{G}}
=\frac{cdG}{1-\frac{1}{\od G}}\left(\frac{1}{\od}-\frac{1}{z}\right),
\end{split}
\end{equation*}
where  $G$ stands for  $\gxy$ for simplicity of notation.
Comparing both sides and noting that we are allowed to cancel out $\frac{1}{\od}-\frac{1}{z}$ we see that
$$\gamma+G=\frac{cdG}{1-\frac{1}{\od G}}.$$
An easy calculation shows that $$\od=\frac{\gamma+\gxy}{\gxy\left(\gamma-(cd-1)\gxy\right)}.$$
Recalling  that $\gxy=G_{\bY}(\od)$, we can write the last equation as
$$\od=\frac{\gamma+G_{\bY}(\od)}{G_{\bY}(\od)\left(\gamma-(cd-1)G_{\bY}(\od)\right)}.$$
This proves that $$G_{\bY}^{-1}(z)=\frac{\gamma+z}{z\left(\gamma-(cd-1)z\right)}.$$
This allows us to determine the r-transform of $\bY$:
$$r_{\bY}(z)=\frac{\gamma+z}{z\left(\gamma-(cd-1)z\right)}-\frac{1}{z}=\frac{cd}{\gamma-(cd-1)z}.$$
\end{proof}
\begin{lemma}
	 The distribution of $\bX+\bY$ free-GIG distribution $\mu\left(\frac{cd}{cd-1},\frac{\gamma}{cd-1},\frac{d}{cd-1}\right)$.
\end{lemma}
\begin{proof}
We already expressed $\od$ in terms of $G=\gxy$ i.e. $\od=\frac{\gamma+G}{G\left(\gamma-(cd-1)G\right)}$.
Plugging  this formula into the second equation of \eqref{ch1sys1} yields the following equation for $G$ in terms of $z$:
\begin{equation*}
(cd-1)z^2G^2-\left(\gamma z^2-z-d\right)G+\gamma z+d\beta=0.
\end{equation*}
By positivity of $\bX$ and $\bY$ we see that $\gamma, \beta, d>0$.  The result follows now from  Lemma \ref{giglemma}.

\end{proof}
\begin{corollary}
The distribution of $\bX$ is free-GIG  $\mu\left(-\frac{cd}{cd-1},\frac{\gamma}{cd-1},\frac{d}{cd-1}\right)$.
\end{corollary}
\begin{proof}
This follows from  Lemma \ref{szpolemm} and the fact that for free and  compactly supported random variables $\bX$ and $\bY$ the distribution of $\bX$ is uniquely determined by distributions of $\bY$ and $\bX+\bY$.
\end{proof}
\end{proof}
\subsection{The case $(k,l)=(1,2)$}

\begin{theorem}	Let $\bX$ and  $\bY$ be  free, positive, self-adjoint random variables. Let us define $\bU=\left(\bX+\bY\right)^{-1}$ and  $\bV=\bX^{-1}-\left(\bX+\bY\right)^{-1}$. If the following conditions are satisfied 
	\begin{equation}\varphi\left(\bV\mid \bU\right)=c\bI,  \end{equation}
	\begin{equation}\varphi\left(\bV^{2}\mid \bU\right)=b\bI,\end{equation}
	for some constants $c$ and $b$, then $b>c^2$ and $\bX$ has the  free-GIG distribution $\mu\left(-\frac{c^2}{b-c^2},\frac{\rho}{b-c^2},\frac{c}{b-c^2}\right)$ and $\bY$  has  the free Poisson distribution $\nu\left(\frac{c^2}{b-c^2},\frac{b-c^2}{\rho}\right)$, where $\rho$ is some positive constant. 
\end{theorem}
\begin{proof}

Let us denote  $\beta=\varphi(\bU)$ and $\alpha=\varphi(\bU^2)$.  Lemma \ref{lemmap1} implies the following equality 
\begin{equation}
\frac{1}{\omega_1(z)}\left(\beta+c+\gxy\right)=\left(c+\frac{1}{z}\right)\gxy+\frac{\beta}{z}.\label{ch2eq1}
\end{equation}
The equation provided by Lemma \ref{lemmap2} is  
	\begin{multline}\label{ch2eq2}
\frac{\varphi(\bX^{-2})}{\oj}+\frac{1}{\oj}\left(\varphi(\bX^{-1})+\gxy\right)\left(\frac{1}{\oj}-\frac{2}{z}\right)\\
=\frac{\varphi(\bX^{-2})-b}{z}-\frac{\beta}{z^2}+\left(b-\frac{1}{z^2}\right)\gxy
\end{multline}
and contains two additional constants $\varphi(\bX^{-1})$ and $\varphi(\bX^{-2})$ that we want to express in terms of $c,b,\alpha$ and $\beta$. First note that equality \eqref{phixneg1} i.e. $\varphi\left(\bX^{-1}\mid \bT\right)=c\bI+\bU$ implies $\varphi(\bX^{-1})=c+\beta$. Combining this with \eqref{l3eq}  yields
\begin{equation*}
\begin{split}
\varphi\left(\bX^{-2}\mid\bT\right)&=2\varphi\left(\bX^{-1}\mid\bT\right)\bT^{-1}+b\bI-\bT^{-2}\\
&=b\bI+2c\bU+\bU^2.
\end{split}
\end{equation*}
Hence $\varphi\left(\bX^{-2}\right)=b+2c\beta+\alpha$. 

Replacing $\frac{1}{\oj}\left(\varphi(\bX^{-1})+\gxy\right)$ in \eqref{ch2eq2} by the right hand side of \eqref{ch2eq1} (and simple algebra) gives the final form of system of  equations we can work with:

\begin{equation}\label{ch2sys} \left\{\begin{array}{rcl}
\frac{1}{\omega_1(z)}\left(\beta+c+\gxy\right)&=&\left(c+\frac{1}{z}\right)\gxy+\frac{\beta}{z}\\
\frac{1}{\omega_1(z)}\left(\delta+\frac{\beta}{z}+\left(c+\frac{1}{z}\right)\gxy\right)&=&\frac{\alpha}{z}+b \gxy+\left(\beta+\gxy\right)\left(\frac{1}{z^2}+\frac{2c}{z}\right)\\
z&=&\omega_1(z)+\omega_2(z)-\frac{1}{\gxy}
\end{array} \right.,
\end{equation}

where  $\delta=b+2c\beta+\alpha$.\\

Note, that this time the first and the second equation in \eqref{ch2sys} do involve $\od$  and we can easily calculate $G=\gxy$. Namely let us  divide the second equation by the first one to get
\begin{equation*}
\frac{\delta+\frac{\beta}{z}+\left(c+\frac{1}{z}\right)G}{\beta+c+G}=\frac{\frac{\alpha}{z}+b G+\left(\beta+G\right)\left(\frac{1}{z^2}+\frac{2c}{z}\right)}{\left(c+\frac{1}{z}\right)G+\frac{\beta}{z}}.
\end{equation*}
(Note  that the expression $\beta+c+G$ is non zero as $G$ takes values in $\mathbb{C}^-$ so this division is justified.)

Multiplying both sides by the denominators we arrive after some easy but tedious calculation at the following equation
\begin{equation}
(b-c^2)z^2G^2-\left(\rho z^2-(2c^2-b)z-c\right)G+\rho z+\beta c=0,\label{ch2eqforg}
\end{equation}
where $\rho=2\beta c^2+\alpha c-\beta b$. Note that $b=\varphi(\bV^2)>c^2=\varphi(\bV)^2$ by the  Cauchy-Schwarz inequality.

Now we are  ready to prove the following lemma.
\begin{lemma}
	The r-transform of $\bY$ is equal
	$$r_{\bY}(z)=\frac{c^2}{\rho-(b-c^2)z},$$
	in particular $\rho>0$ and  $\bY$ has the free Poisson distribution $\nu\left(\frac{c^2}{b-c^2},\frac{b-c^2}{\rho}\right)$.
\end{lemma}
\begin{proof}
From equations \eqref{omega12g}  and \eqref{ch2eq1}
we see that

$$R(z):=\od-\frac{1}{\gxy}=z-\oj=z-\frac{\beta+c+G}{\left(c+\frac{1}{z}\right)G+\frac{\beta}{z}},$$
where $G=\gxy$.

Hence
\begin{equation*}
R(z)=\frac{cz(Gz-1)}{\beta+G+czG}.
\end{equation*}
Now we write  equation \eqref{ch2eqforg} as
\begin{equation*}
(b-c^2)(z^2G^2-Gz)-\rho z(Gz-1)+c^2zG+cG+\beta c=0
\end{equation*}
or 
\begin{equation*}
c(\beta+G+czG)=\rho z(Gz-1)-(b-c^2)zG(zG-1).
\end{equation*}
Multiplying both sides by c and dividing by $\beta +G+czG$ we see that
$$c^2=\rho R(z)-(b-c^2)R(z)G$$ or in other words
$$R(z)=\frac{c^2}{\rho-(b-c^2)G}.$$

Recalling the definition of $R(z)$ and the fact that $G=G_{\bY}(\od)$ we get that
$$\od=\frac{1}{G_{\bY}(\od)}+\frac{c^2}{\rho-(b-c^2)G_{\bY}(\od)}.$$
This proves that $G_{\bY}^{-1}(z)=\frac{1}{z}+\frac{c^2}{\rho-(b-c^2)z}$ and that
$$r_{\bY}(z)=\frac{c^2}{\rho-(b-c^2)z}.$$
Since $r_\bY(0)=\varphi(\bY)>0$ we see that $\rho>0$ and  $\bY$ has the free Poisson distribution $\nu\left(\frac{c^2}{b-c^2},\frac{b-c^2}{\rho}\right)$.
\end{proof}
\begin{corollary} The distribution of $\bX$ is free-GIG $\mu\left(-\frac{c^2}{b-c^2},\frac{\rho}{b-c^2},\frac{c}{b-c^2}\right)$ where $\rho=2\beta c^2+\alpha c-\beta b$.
\end{corollary}
\begin{proof}
The Cauchy-Stieltjes transform $G=\gxy$ satisfies the quadratic equation \eqref{ch2eqforg} i.e.
\begin{equation*}
(b-c^2)z^2G^2-\left(\rho z^2-(2c^2-b)z-c\right)G+\rho z+\beta c=0.
\end{equation*}
Since we know that $\rho>0$  Lemma \ref{giglemma} implies that the distribution of $\bX+\bY$ is $\mu\left(\frac{c^2}{b-c^2},\frac{\rho}{b-c^2},\frac{c}{b-c^2}\right)$. The result follows from Lemma \ref{szpolemm}.
\end{proof}
	\end{proof}
\subsection{The case $(k,l)=(1,2)$}
\begin{theorem}
	Let  $\bX$ and  $\bY$ be free positive non-commutative random variables. Let us define $\bU=\left(\bX+\bY\right)^{-1}$ and  $\bV=\bX^{-1}-\left(\bX+\bY\right)^{-1}$. If the following conditions are satisfied 
	\begin{equation}\varphi\left(\bV^{-1}\mid \bU\right)=d\bI,\label{cond33}  \end{equation}
	\begin{equation}\varphi\left(\bV^{-2}\mid \bU\right)=h\bI,\end{equation}
	for some constants $d$ and $h$, then $h>d^2$ and $\bX$ has the free-GIG distribution  $\mu\left(-\frac{h}{h-d^2},\frac{\gamma d^2}{h-d^2}, \frac{ d^3}{h-d^2}\right)$ and $\bY$ has the free Poisson distribution $\nu\left(\frac{h}{h-d^2},\frac{h-d^2}{d^2\gamma}\right)$, where $\gamma$ is some positive constant.
\end{theorem}
\begin{proof}
	Let us denote   $\gamma=\varphi(\bY^{-1})$, $\beta=\varphi(\bU)$ and  $\alpha=\varphi(\bU^2)$.
	Lemma \ref{lemmam1} implies the following equality
	\begin{equation*}
	\frac{1}{\omega_2(z)}\left(\gamma+\gxy\right)=\frac{d\beta}{z^2}+\frac{\gamma}{z}+\left(\frac{d}{z^2}+\frac{1}{z}\right)\gxy
	\end{equation*}
that can be written also as
\begin{equation}
\left(\gamma+\gxy\right)\left(\frac{1}{\od}-\frac{1}{z}\right)=\frac{d}{z^2}\left(\beta+\gxy\right).\label{MY3Eq1}
\end{equation}
From Lemma \ref{lemmam2} we get the second equation:
	\begin{multline}\label{ch3eq1}
\varphi(\bX^2)B(z)+A(z)^2\left(\omega_1^2(z)\gxy-\omega_1(z)-\varphi(\bX)\right)=\\
=h\left(\frac{\alpha}{z}+\frac{\beta}{z^2}+\frac{1}{z^2}\gxy\right),
\end{multline}
where
$$A(z)=\frac{\gamma+\gxy}{\omega_2(z)\gxy}\ \textrm{and} \ \ B(z)=\frac{\varphi(\bY^{-2})-\gamma A(z)}{\omega_2(z)}.$$

Our first  goal is to simplify  equation \eqref{ch3eq1}. First note that $\bV^{-1}\bU=\bX\bY^{-1}$. Since $\bX$ and $\bY^{-1}$ are free we get
$$\varphi(\bX)\varphi(\bY^{-1})=\varphi(\bV^{-1}\bU)=\varphi(\varphi(\bV^{-1}\mid\bU)\bU)=d\beta$$
Hence $\varphi(\bX)=\frac{d\beta}{\gamma}$. Similarly $\bV^{-1}=\bX\bY^{-1}(\bX+\bY)=\bX\bY^{-1}\bX+\bX$. Taking expectation we see that $$\varphi(\bV^{-1})=\varphi(\bX\bY^{-1}\bX)+\varphi(\bX)=\varphi(\bX^2\bY^{-1})+\varphi(\bX)$$
by traciality. From this we get $\varphi(\bX^2)\gamma=d\left(1-\frac{\beta}{\gamma}\right)$. Next note that equation \eqref{lm2eq1} i.e.
\begin{equation*}
\varphi\left(\bY^{-1}\bX^{2}\bY^{-1}\mid\bU\right)=h\bU^{2}
\end{equation*} implies that $\varphi(\bX^2)\varphi(\bY^{-2})=h\alpha$.
Taking this into account we see that
$$\varphi(\bX^2)B(z)=\frac{\varphi(\bX^2)\varphi(\bY^{-2})-\varphi(\bX^2)\gamma A(z)}{\omega_2(z)}=\frac{h\alpha}{\od}-d\left(1-\frac{\beta}{\gamma}\right)\frac{A(z)}{\od}$$
Thus the left hand side of \eqref{ch3eq1} is equal
$$\frac{h\alpha}{\od}-d\left(1-\frac{\beta}{\gamma}\right)\frac{A(z)}{\od}+A(z)^2\left(\omega_1^2(z)\gxy-\omega_1(z)\right)-\frac{d\beta}{\gamma}A(z)^2$$

An easy calculation shows that $$d\left(1-\frac{\beta}{\gamma}\right)\frac{A(z)}{\od}+\frac{d\beta}{\gamma}A(z)^2=A(z)\frac{d(\beta+\gxy)}{\od\gxy}$$
and that
$$\omega_1^2(z)\gxy-\omega_1(z)=\oj\gxy\left(\oj-\frac{1}{\gxy}\right)=\gxy\oj\left(z-\od\right),$$
where in the last equality we used formula \eqref{omega12g}.

Consequently, equation \eqref{ch3eq1} takes on the following form
\begin{multline*}
\frac{h\alpha}{\od}+\gxy\oj\left(z-\od\right)A(z)^2-A(z)\frac{d(\beta+\gxy)}{\od\gxy}=\\
h\left(\frac{\alpha}{z}+\frac{\beta}{z^2}+\frac{1}{z^2}\gxy\right)
\end{multline*}
or equivalently
\begin{multline*}
h\alpha\left(\frac{1}{\od}-\frac{1}{z}\right)+z\gxy\oj\od\left(\frac{1}{\od}-\frac{1}{z}\right)A(z)^2=\\
h\frac{\beta+\gxy}{z^2}+A(z)\frac{d(\beta+\gxy)}{\od\gxy}.
\end{multline*}
We can now plug $\beta+\gxy$ calculated from \eqref{MY3Eq1} and cancel out the common term i.e. $\frac{1}{\od}-\frac{1}{z}$. The cancellation is allowed since both sides of the equation are analytic on $\mathbb{C}^+$ and $\od$ cannot be the identity function as it would contradict positivity of $\bX$. This yields  the following simpler equation
\begin{equation}\label{ch3eq1.5}
\begin{split}
h\alpha+z\oj\od\gxy A(z)^2&=\frac{h}{d}(\gamma+\gxy)+z^2A(z)\frac{\gamma+\gxy}{\od\gxy}\\
&=\frac{h}{d}(\gamma+\gxy)+z^2A(z)^2.
\end{split}
\end{equation}
To proceed further we need to express $\alpha$ in terms of other constants. Note that
$$\bY^{-1}-\bU\bV^{-1}\bU=\bU\bU^{-1}\bY^{-1}-\bU\bX\bY^{-1}=\bU(\bU^{-1}-\bX)\bY^{-1}=\bU.$$
After taking expectation and using regression condition \eqref{cond33} we see that $\gamma-d\alpha=\beta$. In other words 
$$\alpha=\frac{\gamma-\beta}{d}.$$
This fact and  \eqref{MY3Eq1} imply that
$$\frac{h}{d}(\gamma+\gxy)-h\alpha=h\frac{\beta+\gxy}{d}=\frac{hz^2}{d^2}\left(\gamma+\gxy\right)\left(\frac{1}{\od}-\frac{1}{z}\right).$$
The right hand side of the above expression can be written as $$\frac{hz\gxy}{d^2}A(z)(z-\od).$$
This means that   we can rewrite equation \eqref{ch3eq1.5} as
$$z\oj\od\gxy A(z)^2=\frac{hz\gxy}{d^2}A(z)(z-\od)+z^2A(z)^2$$
or after canceling out the common term $zA(z)\neq 0$
$$\oj\od\gxy A(z)=zA(z)+\frac{h}{d^2}\gxy(z-\od).$$
Equation \eqref{omega12g} implies that $\oj\gxy=\gxy(z-\od)+1$. Plugging this into the above equation we get
$$\gxy(z-\od)\od A(z)=(z-\od)A(z)+\frac{h}{d^2}\gxy(z-\od.)$$

Since $z-\od$ is a non zero function  and both sides are analytic on $\mathbb{C}^+$ we obtain 
\begin{equation}\label{ch3eq2}
 \gamma+\gxy=\frac{\gamma+\gxy}{\od\gxy}+\frac{h}{d^2}\gxy.
\end{equation}

\begin{lemma}
	The r-transform of $\bY$ is equal
	$$r_{\bY}(z)=\frac{h}{d^2\gamma-(h-d^2)z},$$
 and hence $\bY$ has the  free Poisson distribution $\nu\left(\frac{h}{h-d^2},\frac{h-d^2}{d^2\gamma}\right)$.
\end{lemma}
\begin{proof}
Equation \eqref{ch3eq2} implies
\begin{equation*}
\od=\frac{d^2(\gamma+\gxy)}{\gxy\left(d^2\gamma-(h-d^2)\gxy\right)}.
\end{equation*}
Since $\gxy=G_{\bY}(\od)$ we see that
$$G^{-1}_\bY(z)=\frac{d^2(\gamma+z)}{z\left(d^2\gamma-(h-d^2)z\right)}.$$
Thus  $$r_{\bY}(z)=G^{-1}_\bY(z)-\frac{1}{z}=\frac{h}{d^2\gamma-(h-d^2)z}.$$
Now it is enough to note that the  Cauchy-Schwarz inequality implies $h>d^2$.
\end{proof}
\begin{lemma}
	The distribution of $\bX$   is free-GIG   $\mu\left(-\frac{h}{h-d^2},\frac{\gamma d^2}{h-d^2}, \frac{ d^3}{h-d^2}\right)$.

\end{lemma}
\begin{proof}
Equations \eqref{MY3Eq1}  and \eqref{ch3eq2} imply the following quadratic equation for $G=\gxy$
\begin{equation*}
(h-d^2)z^2G^2-d^2(\gamma z^2-z-d)G+d^2(\gamma z+d\beta)=0.
\end{equation*}
Since all parameters are obviously positive and $h>d^2$ we see that	Lemma \ref{giglemma} implies that the distribution of $\bX+\bY$ is free-GIG $\mu\left(\frac{h}{h-d^2},\frac{\gamma d^2}{h-d^2}, \frac{ d^3}{h-d^2}\right)$.

  The result follows  now from Lemma \ref{szpolemm}.
\end{proof}
\end{proof}
\subsection*{Acknowledgment} The author thanks J. Weso\l{}owski and K. Szpojankowski for helpful comments and discussions.
\bibliographystyle{plain}
\bibliography{MarcinSwiecaMYbib} 

\begin{thebibliography}{10}

\bibitem{bao2021characterizations}
Kevin~B Bao and Christian Noack.
\newblock Characterizations of the generalized inverse {G}aussian, asymmetric
  {L}aplace, and shifted (truncated) exponential laws via independence
  properties.
\newblock {\em arXiv preprint arXiv:2107.01394}, 2021.

\bibitem{belinschi2007new}
S.~T. Belinschi and H.~Bercovici.
\newblock A new approach to subordination results in free probability.
\newblock {\em J. Anal. Math.}, 101:357--365, 2007.

\bibitem{belinschi2017outliers}
Serban~T. Belinschi, Hari Bercovici, Mireille Capitaine, and Maxime
  F\'{e}vrier.
\newblock Outliers in the spectrum of large deformed unitarily invariant
  models.
\newblock {\em Ann. Probab.}, 45(6A):3571--3625, 2017.

\bibitem{belinschi2006note}
Serban~Teodor Belinschi.
\newblock A note on regularity for free convolutions.
\newblock {\em Ann. Inst. H. Poincar\'{e} Probab. Statist.}, 42(5):635--648,
  2006.

\bibitem{belinschi2008lebesgue}
Serban~Teodor Belinschi.
\newblock The {L}ebesgue decomposition of the free additive convolution of two
  probability distributions.
\newblock {\em Probab. Theory Related Fields}, 142(1-2):125--150, 2008.

\bibitem{biane1998processes}
Philippe Biane.
\newblock Processes with free increments.
\newblock {\em Math. Z.}, 227(1):143--174, 1998.

\bibitem{MR2091757}
Chao-Wei Chou and Wen-Jang Huang.
\newblock On characterizations of the gamma and generalized inverse {G}aussian
  distributions.
\newblock {\em Statist. Probab. Lett.}, 69(4):381--388, 2004.

\bibitem{ejsmont2017convolution}
W.~Ejsmont, U.~Franz, and K.~Szpojankowski.
\newblock Convolution, subordination, and characterization problems in
  noncommutative probability.
\newblock {\em Indiana Univ. Math. J.}, 66(1):237--257, 2017.

\bibitem{feral2006limiting}
Delphine F\'{e}ral.
\newblock The limiting spectral measure of the generalised inverse {G}aussian
  random matrix model.
\newblock {\em C. R. Math. Acad. Sci. Paris}, 342(7):519--522, 2006.

\bibitem{fevrier2020using}
Maxime Fevrier, Mitja Mastnak, Alexandru Nica, and Kamil Szpojankowski.
\newblock Using {B}oolean cumulants to study multiplication and
  anti-commutators of free random variables.
\newblock {\em Trans. Amer. Math. Soc.}, 373(10):7167--7205, 2020.

\bibitem{kolodziejek2017matsumoto}
Bartosz Ko{\l}odziejek.
\newblock The {M}atsumoto-{Y}or property and its converse on symmetric cones.
\newblock {\em J. Theoret. Probab.}, 30(2):624--638, 2017.

\bibitem{lehner2019boolean}
Franz Lehner and Kamil Szpojankowski.
\newblock Boolean cumulants and subordination in free probability.
\newblock {\em Random Matrices: Theory Appl}, To appear.

\bibitem{letac2000independence}
G\'{e}rard Letac and Jacek Weso{\l}owski.
\newblock An independence property for the product of {GIG} and gamma laws.
\newblock {\em Ann. Probab.}, 28(3):1371--1383, 2000.

\bibitem{MassamWeso}
H\'{e}l\`ene Massam and Jacek Weso{\l}owski.
\newblock The {M}atsumoto-{Y}or property on trees.
\newblock {\em Bernoulli}, 10(4):685--700, 2004.

\bibitem{matsumoto2001analogue}
Hiroyuki Matsumoto and Marc Yor.
\newblock An analogue of {P}itman's {$2M-X$} theorem for exponential {W}iener
  functionals. {II}. {T}he role of the generalized inverse {G}aussian laws.
\newblock {\em Nagoya Math. J.}, 162:65--86, 2001.

\bibitem{mingofree}
James~A. Mingo and Roland Speicher.
\newblock {\em Free probability and random matrices}, volume~35 of {\em Fields
  Institute Monographs}.
\newblock Springer, New York; Fields Institute for Research in Mathematical
  Sciences, Toronto, ON, 2017.

\bibitem{szpojankowski2017matsumoto}
Kamil Szpojankowski.
\newblock On the {M}atsumoto-{Y}or property in free probability.
\newblock {\em J. Math. Anal. Appl.}, 445(1):374--393, 2017.

\bibitem{szpojankowski2019conditional}
Kamil Szpojankowski and Jacek Weso{\l}owski.
\newblock Conditional expectations through {B}oolean cumulants and
  subordination---towards a better understanding of the {L}ukacs property in
  free probability.
\newblock {\em ALEA Lat. Am. J. Probab. Math. Stat.}, 17(1):253--272, 2020.

\bibitem{voiculescu1986addition}
Dan Voiculescu.
\newblock Addition of certain noncommuting random variables.
\newblock {\em J. Funct. Anal.}, 66(3):323--346, 1986.

\bibitem{voiculescu2000coalgebra}
Dan Voiculescu.
\newblock The coalgebra of the free difference quotient and free probability.
\newblock {\em Internat. Math. Res. Notices}, (2):79--106, 2000.

\bibitem{MR1888812}
Jacek Weso\l{}owski.
\newblock The {M}atsumoto-{Y}or independence property for {GIG} and gamma laws,
  revisited.
\newblock {\em Math. Proc. Cambridge Philos. Soc.}, 133(1):153--161, 2002.

\end{thebibliography}
\end{document}